\documentclass[a4paper,11pt]{article}

\usepackage{amssymb,amsmath,amsthm,times,mathrsfs,fullpage}
\usepackage{tikz,wrapfig,multirow}
\usepackage[pdftex]{hyperref}
\hypersetup{plainpages=True, pdfstartview=FitH, bookmarksopen=true, colorlinks=true,linkcolor=blue,citecolor=blue}
\usepackage{enumitem}

\renewcommand{\S}{\mathcal{S}}

\usetikzlibrary{decorations.markings}

\begin{document}

\setlist[description]{font=\normalfont\space}

\pgfdeclarelayer{background}
\pgfdeclarelayer{foreground}
\pgfsetlayers{background,main,foreground}

\newtheorem{thm}{Theorem}
\newtheorem{cor}{Corollary}
\newtheorem{lmm}{Lemma}
\newtheorem{conj}{Conjecture}
\newtheorem{pro}{Proposition}
\newtheorem{Def}{Definition}
\theoremstyle{remark}\newtheorem{Rem}{Remark}

\title{Bijections for Ranked Tree-Child Networks}
\author{Alessandra Caraceni\\
    Istituto Nazionale di Alta Matematica\\
    Unit\`a di ricerca SNS\\
    P.zza dei Cavalieri 7, Pisa\\
    Italy \and
    Michael Fuchs\thanks{Supported by MOST under the research grant MOST-109-2115-M-004-003-MY2.}\\
    Department of Mathematical Sciences\\
    National Chengchi University\\
    Taipei 116\\
    Taiwan \and
    Guan-Ru Yu\\
    Department of Mathematics\\
    National Kaohsiung Normal University\\
    Kaohsiung 824\\
    Taiwan}

\maketitle

\begin{abstract}
The class of ranked tree-child networks,  tree-child networks arising from an evolution process with a fixed embedding into the plane, has recently been introduced by Bienvenu, Lambert, and Steel. These authors derived counting results for this class. In this note, we will give bijective proofs of three of their results. Two of our bijections answer questions raised in their paper.
\end{abstract}

\section{Introduction}\label{intro}

Phylogenetic trees and networks are important discrete structures from biology where they are used to model evolution; see \cite{HuRuSc,SeSt,St}. Of these two types of structures, phylogenetic trees are simpler and more classical but they are less suitable to model evolutionary scenarios that involve reticulation events. Thus, in many recent studies, they have been replaced by (the more general) phylogenetic networks. However, the majority of the classes of phylogenetic networks are not recursive and thus they are a poor model for processes that evolve over time. In order to be able to model such processes, Bienvenu et al.~recently proposed the class of {\it ranked tree-child networks}; see \cite{BiLaSt}.

We recall some definitions. First, a (rooted, binary) {\it phylogenetic network} is a simple DAG (directed acyclic graph) with a unique root of indegree $0$ and outdegree $1$ such that all other nodes belong to one of the following types:
\begin{itemize}
\item {\it leaves}, which are nodes with indegree $1$ and outdegree $0$; these are bijectively labeled by the set $\{1,\ldots,\ell\}$ where $\ell$ is the total number of leaves;
\item {\it tree nodes}, which are nodes with indegree $1$ and outdegree $2$; and
\item {\it reticulation nodes}, which are nodes with indegree $2$ and outdegree $1$.
\end{itemize}
A phylogenetic network is called a {\it tree-child} network if each internal (i.e., non-leaf) node has at least one child which is not a reticulation node. For such a network, we call a tree node whose two children are both not reticulation nodes a {\it branching event} and a reticulation node together with both its parents a {\it reticulation event}; see Figure~\ref{events}-(a). A {\it ranked tree-child network} (or RTCN for short) is now a tree-child network which is drawn in such a way that one starts with a branching event and then successively adds either a branching event or a reticulation event until the leaves are reached; see Figure~\ref{events}-(b).

\vspace*{0.25cm}
\begin{figure}[t!]
\centering
\includegraphics[scale=0.9]{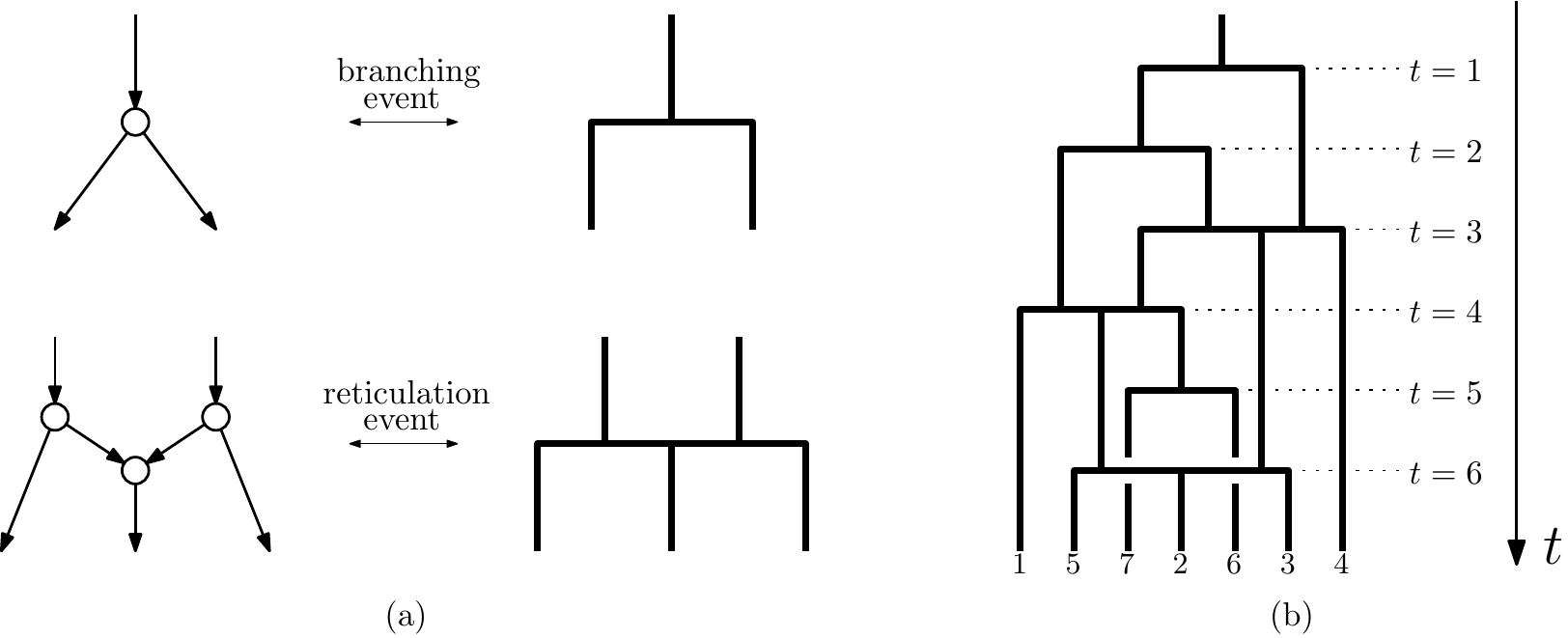}
\caption{(a) The two types of events of a ranked tree-child network; (b) An example of a RTCN.}\label{events}
\end{figure}

The class of RTCNs has the advantage over other network classes that the special embedding into the plane makes counting relatively straightforward. (In contrast, few of the other classes of phylogenetic networks have so far been counted; see, e.g., \cite{BoGaMa,DiSeWe,FuGiMa1,FuGiMa2,FuHuYu,FuYuZh1,FuYuZh2} for some recent progress.) For instance, the following simple formula was obtained in \cite{BiLaSt} for the number $\mathrm{RTC}_{\ell}$ of RTCNs with $\ell$ leaves:
\begin{equation}\label{num-of-RTCN}
\mathrm{RTC}_{\ell}=\frac{\ell!(\ell-1)!^2}{2^{\ell-1}}.
\end{equation}
Curiously, the same number is the answer to the following counting problem: find the number of ways for $\ell$ people to cross a river with a two-person boat where the boat trips follow the pattern $2$ send, $1$ returns, $2$ send, $1$ returns, etc; see A167484 in the OEIS. Thus, one can ask for a bijection between RTCNs and such boat sequences; however, in Remark 2.8 in \cite{BiLaSt} it was claimed that a natural bijection seems to be unlikely because for $\ell=3$, $3$ out of the $6$ RTCNs contain a reticulation event and $3$ do not (see bottom resp. top row of Figure~\ref{bij-l-3}) whereas all $6$ possible boat trips are completely equivalent. Nevertheless, we will give a simple bijection in the next section; again see Figure~\ref{bij-l-3} for the correspondence between RTCNs and boat sequences for $\ell=3$ arising from our bijection.

\begin{figure}[h!]
\centering
\includegraphics[scale=0.9]{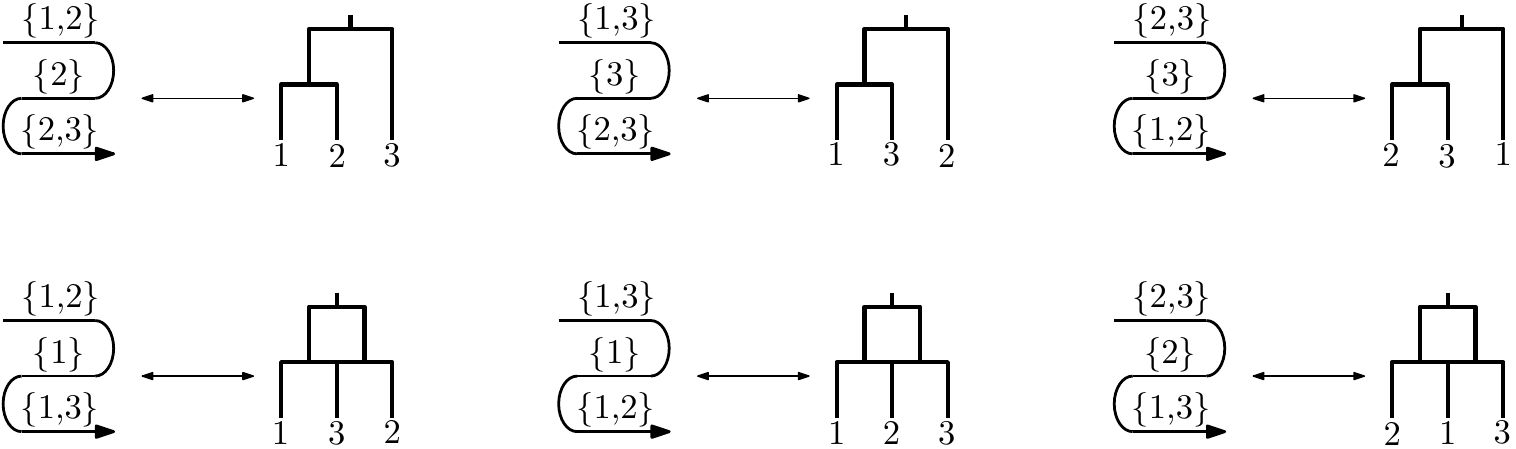}
\caption{All RTCNs and their corresponding boat sequences for $\ell=3$.}\label{bij-l-3}
\end{figure}

Using this bijection, branching events are, e.g., mapped to the following events of boat sequences. Assume that the $\ell$ people are ranked according to their skill at steering the boat. Then, the number of branching events corresponds to the number $X_{\ell}$ of return trips where the most skilled person among those on the opposite shore takes the boat back. Since it was proved in Corollary 2.7 in \cite{BiLaSt} that the number of branching events in a RTCN chosen uniformly at random from all RTCNs with $\ell$ leaves satisfies a central limit theorem, we obtain the following result.

\begin{thm}\label{only-thm}
For a boat sequence of $\ell$ people chosen uniformly at random from all boat sequences, the number of times $X_{\ell}$ that the most skilled person takes the boat back satisfies the limit law:
\[
\frac{X_{\ell}-\log\ell}{\sqrt{\log\ell}}\stackrel{d}{\longrightarrow}N(0,1),
\]
where $\stackrel{d}{\longrightarrow}$ denotes convergence in distribution and $N(0,1)$ denotes the standard normal distribution.
\end{thm}

In fact, the number $\mathrm{RTC}_{\ell,b}$ of RTCNs with $\ell$ leaves and $b$ branching events was counted in \cite{BiLaSt} as well:
\[
\mathrm{RTC}_{\ell,b}=\genfrac{[}{]}{0pt}{}{\ell-1}{b}\cdot\mathrm{RT}_{\ell},
\]
where the bracket denotes the (signless) Stirling numbers of the first kind and $\mathrm{RT}_{\ell}=\ell!(\ell-1)!/2^{\ell-1}$ is the number of ranked trees (i.e., RTCNs without reticulation events). Since the Stirling numbers of the first kind count permutations with a fixed number of cycles, this result suggests that there should be a simple bijection between RTCNs with a fixed number of branching events and pairs consisting of permutations with a fixed number of cycles and ranked trees. We will give such a bijection in Section~\ref{RTCN-fixed-b}.

\begin{figure}[t!]
\centering
\includegraphics{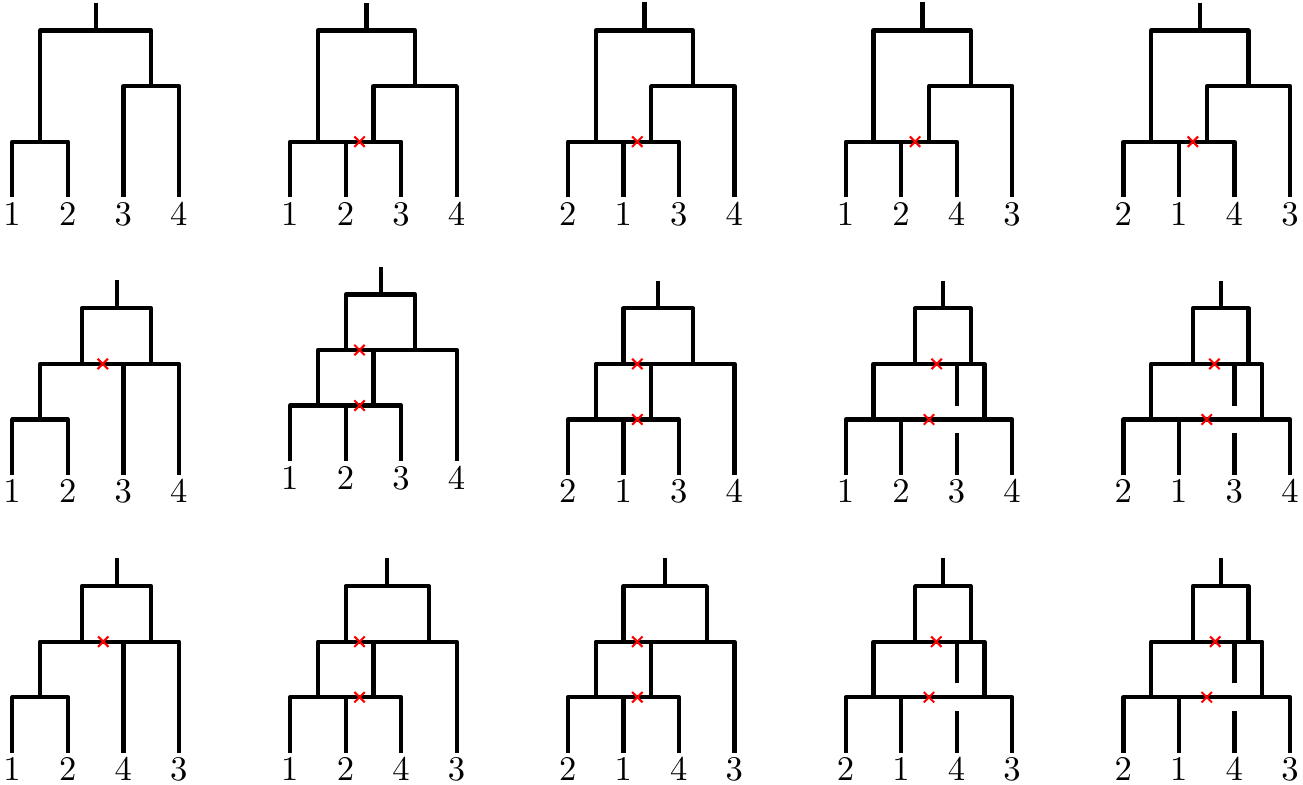}
\caption{A ranked tree $T$ (upper left corner) with $\ell=4$ leaves and all $15$ RTCNs which contain $T$ where incoming edges of reticulation nodes which need to be removed to obtain $T$ are indicated by arrows. (Note that one of the RTCNs is $T$ itself.)}
\label{tree-cont}
\end{figure}

A final bijection discussed in this note is related to another result in \cite{BiLaSt} which concerns containment of ranked trees in RTCNs. Here, we say that a ranked tree $T$ is contained in a RTCN $N$, in symbols $T\sqsubset N$, if $T$ can be obtained from $N$ by choosing one of the incoming edges of each reticulation node of $N$, removing them, and then suppressing resulting nodes with indegree $1$ and outdegree $1$; see Figure~\ref{tree-cont}.

For a fixed ranked tree $T$ with $\ell$ leaves, it was proved in \cite{BiLaSt} that
\begin{equation}\label{num-tree-contain}
\#\{N\ :\ T\sqsubset N\}=1\cdot 3\cdot 5\cdots(2\ell-3)=:(2\ell-3)!!.
\end{equation}
Note that $(2\ell-3)!!$ is the number of phylogenetic trees with $\ell$ leaves where a phylogenetic tree is a phylogenetic network without reticulation nodes; see, e.g., Corollary 2.2.4 in \cite{SeSt}. Thus, one again would like to have a simple bijection between RTCNs which contain $T$ and phylogetetic trees; see Remark 3.7 in \cite{BiLaSt}. In the final section of this note, we will give such a bijection.

\section{RTCNs and boat sequences}

In this section, we will describe a bijection between RTCNs with $\ell$ leaves and boat sequences involving $\ell$ people. This will be done in three steps. First, we will explain how to construct all RTCNs with $\ell+1$ leaves from those with $\ell$ leaves; then, we will do the same for boat sequences; and finally we will synchronize these two constructions to get the desired bijection.

\paragraph{Recursive construction of RTCNs.} Here, we give a top-down construction which (uniquely) produces all RTCNs with $\ell+1$ leaves. A similar construction was described in \cite{BiLaSt} where the authors used the concept of decorated RTCNs. We will sidestep this concept which gives a (slightly) simpler construction.

\begin{figure}
\centering
\includegraphics[scale=0.7]{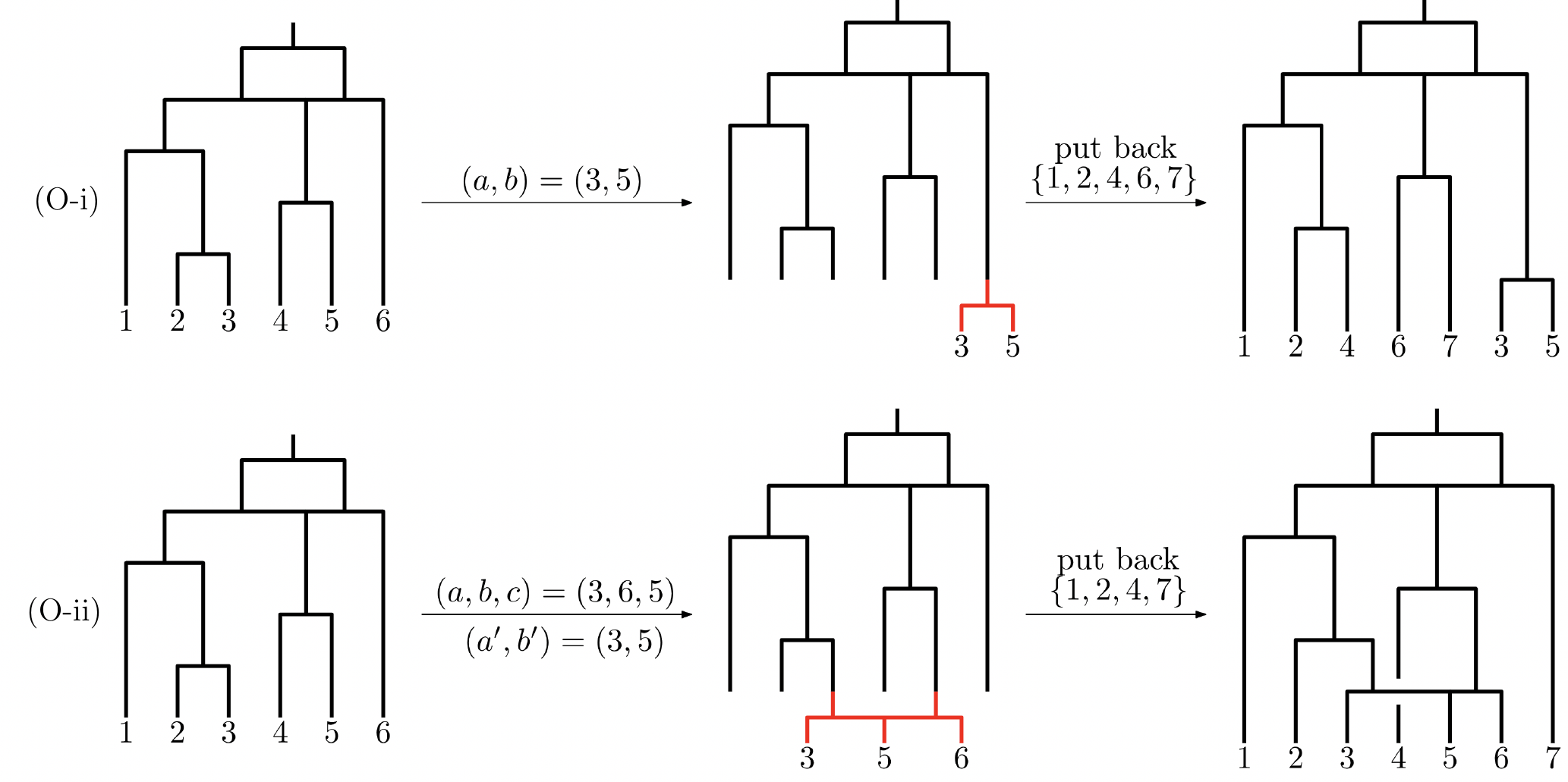}
\caption{The two operations (O-i) and (O-ii) in the recursive construction of RTCNs.}\label{op-rtcn}
\end{figure}

Assume we have a given RTCN with $\ell$ leaves. Then, we perform either one of the following two operations (see Figure~\ref{op-rtcn} for examples).
\begin{description}
\item[(O-i)] Pick two labels $a,b$ from the set $\{1,\ldots,\ell+1\}$, remove the label $\ell$ from the RTCN and attach to its leaf two children with labels $a,b$. Finally, relabel all the other leaves of the RTCN in an order-consistent way with the remaining labels from $\{1,\ldots,\ell+1\}\setminus\{a,b\}$.
\item [(O-ii)] Pick two labels $a<b$ from $\{1,\ldots,\ell+1\}$ and a label $c$ from $\{1,\ldots,\ell+1\}\setminus\{a,b\}$. Next, find the relative ranks, say $a'<b'$, of $a<b$ in $\{1,\ldots,\ell+1\}\setminus\{c\}$. Remove $a',b'$ from the RTCN and attach a reticulation event to their leaves with the children corresponding to the leaf which had the label $a'$ being a leaf with label $a$ and the reticulation vertex, the children corresponding to the leaf which had the label $b'$ being a leaf with label $b$ and the (same) reticulation vertex and the child of the reticulation vertex being a leaf with label $c$. Finally, relabel all the other leaves of the RTCN in an order-consistent way with the labels from $\{1,\ldots,\ell+1\}\setminus\{a,b,c\}$.
\end{description}

Since both of these operations are reversible, we have the following result.
\begin{pro}\label{rec-rtcn}
Starting from all RTCNs with $\ell$ leaves and performing for each of them all the operations above gives (exactly once) all the RTCNs with $\ell+1$ leaves.
\end{pro}

Moreover, this construction also immediately gives (\ref{num-of-RTCN}).

\begin{cor}
We have
\[
\mathrm{RTC}_{\ell}=\frac{\ell!(\ell-1)!^2}{2^{\ell-1}}.
\]
\end{cor}
\begin{proof}
The number of different ways to perform (i) and (ii) above is
\[
\binom{\ell+1}{2}+\binom{\ell+1}{2}(\ell-1)=\frac{(\ell+1)\ell^2}{2}.
\]
Thus,
\[
\mathrm{RTC}_{\ell}=\prod_{j=1}^{\ell-1}\frac{(\ell+1)\ell^2}{2}=\frac{\ell!(\ell-1)!^2}{2^{\ell-1}}
\]
which proves the claim.
\end{proof}

\begin{figure}[h!]
\centering
\includegraphics{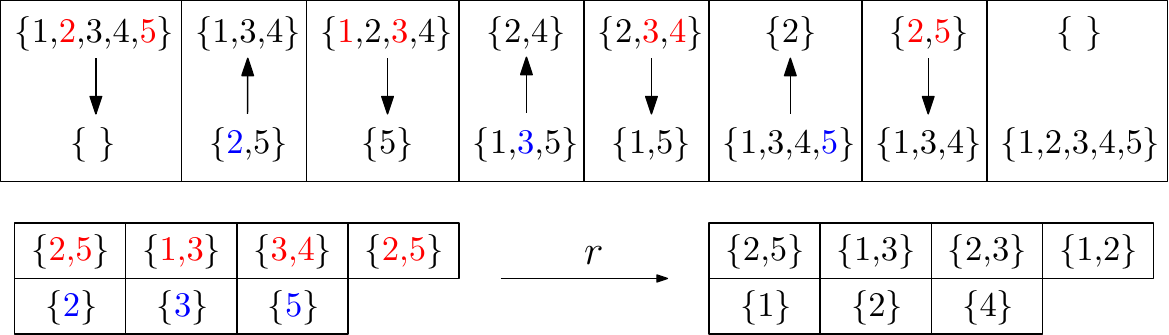}
\caption{An example of a boat sequence involving $5$ people. Top: the people on both sides of the shores after each step. Bottom: the representation of the boat sequence as array and the array of rankings obtained from it by applying the map $r$.}
\label{boat-to-arrays}
\end{figure}

\vspace*{-0.1cm}
\paragraph{Recursive construction of boat sequences.} Here, we explain how all boat sequences involving $\ell+1$ people can be constructed from those involving $\ell$ people. (Note that this reveals a recursive structure inherent to boat sequences that was not considered in \cite{BiLaSt}.)

First, note that a boat sequence involving $\ell$ people can be represented by an array with $2$ rows where the first row has $\ell-1$ entries and the second has $\ell-2$ entries, namely, the first row contains the set of numbers of the people sent in the $i$-th step and the second row contains the singleton of the number of the person who returned in the $i$-th step; see bottom left array in Figure~\ref{boat-to-arrays} for an example.

Now, each such array can be mapped to a corresponding array where the numbers are replaced by the relative ranks of the people within the group on their shore, e.g., if $\{2,5\}$ are sent with $\{2,3,5,6\}$ at the same shore, then $\{2,5\}$ is replaced by $\{1,3\}$. Call the resulting map $r$; see bottom of Figure~\ref{boat-to-arrays}.

Note that the map $r$ is clearly a bijection from the above arrays representing boat sequences to arrays where in the first row, we have a subset of size $2$ of $\{1,\ldots,\ell\}$, followed by a subset of size $2$ of $\{1,\ldots,\ell-1\}$, etc. until the set $\{1,2\}$ and in the second row, we have a subset of size $1$ of $\{1,2\}$, followed by a subset of size $1$ of $\{1,2,3\}$, etc. until a subset of size $1$ of $\{1,\ldots,\ell-1\}$.

Now, we apply the following operations on the latter arrays: we shift all entries of the first row to the right by one position, add a subset of size $2$ of $\{1,\ldots,\ell+1\}$ in the (now empty) first position and add a subset of size $1$ of $\{1,\ldots,\ell\}$ at the end of the second row; see the downwards arrow in Figure~\ref{rec-boat}.

Finally, by using the inverse of $r$, we obtain an array for a boat sequence involving $\ell+1$ people; see the bottom left array in Figure~\ref{rec-boat}.

\begin{figure}
\centering
\includegraphics{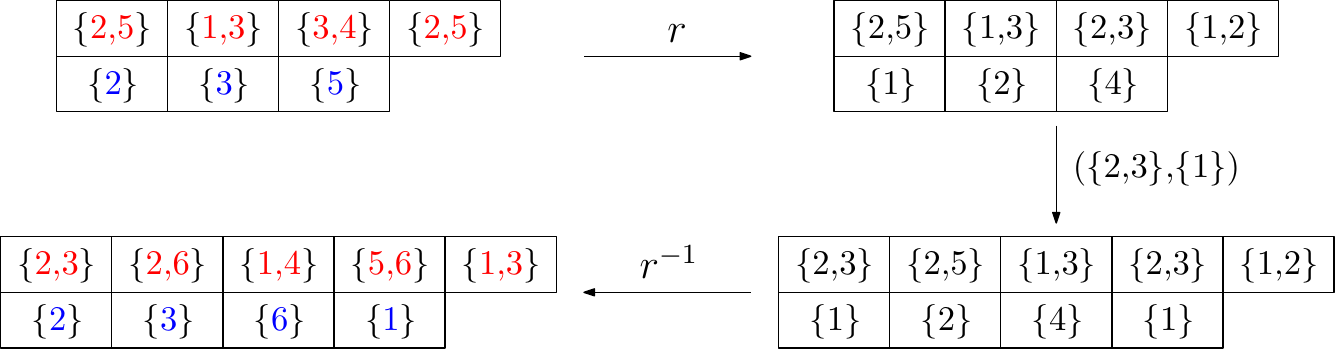}
\caption{A boat sequence with $6$ people constructed from the boat sequence with $5$ people from Figure~\ref{boat-to-arrays}.}
\label{rec-boat}
\end{figure}

Since the above process is reversible, we obtain the following result.

\begin{pro}
Starting from a boat sequence involving $\ell$ people and using the map $r$, all possible operations above, and the inverse of $r$, we obtain (exactly once) all boat sequences involving $\ell+1$ people.
\end{pro}

\begin{figure}[h!]
\centering
\includegraphics[scale=0.9]{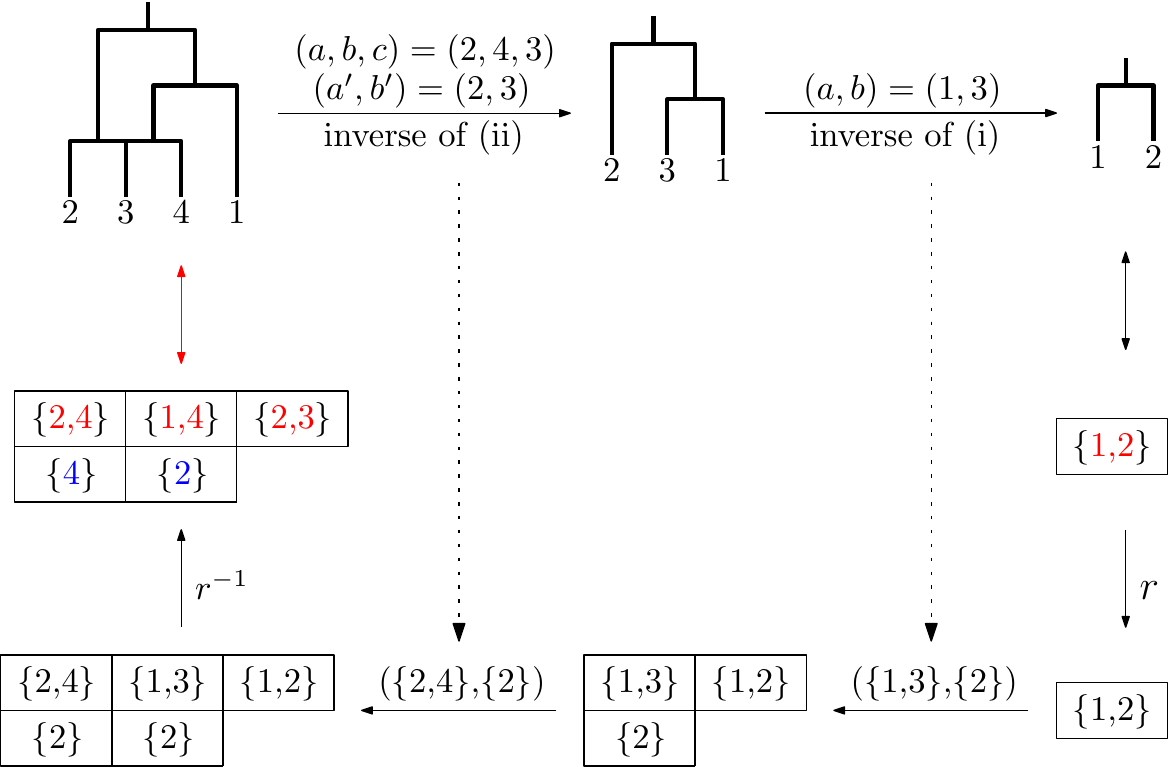}
\caption{A RTCN (top left corner) with $\ell=4$ and the process of constructing its image under the bijection. Top: the two steps of reducing the RTCN to the one with $\ell=2$. Right-most column: The boat sequence corresponding to the initial case and its array of rankings. Bottom: the process of extending the array of rankings. Left-most column: applying $r^{-1}$ gives the desired result (middle).}
\label{RTCN-to-boat}
\end{figure}

\paragraph{Synchronization of the two constructions.} We will now explain how to synchronize the constructions from the above two paragraphs to get a bijection between RTCNs and boat sequences. This bijection is defined recursively.

Assume we have given a RTCN on $\ell+1$ leaves. Reverse the recursive construction from Proposition~\ref{rec-rtcn} to obtain a RTCN on $\ell$ leaves. By induction hypothesis, this RTCN is mapped on a boat sequence which is represented by an array as described in the previous paragraph. We apply the map $r$ to this array. Now, depending on which of the two operations was used in Proposition~\ref{rec-rtcn} to construct the RTCN on $\ell+1$ leaves from that of $\ell$ leaves, we do the following to the array of ranks.
\begin{description}
\item[(i)] If Operation (O-i) was used, then move all the pairs from the first row by one position and add $\{a,b\}$ as the first entry. Also, add $\{\ell\}$ at the end of the second row.
\item[(ii)] If Operation (O-ii) was used, then again move all the pairs from the first row by one position and add $\{a,b\}$ as the first entry. Add the relative rank of $c$ in $\{1,\ldots,\ell+1\}\setminus\{a,b\}$ to the end of the second row.
\end{description}

Finally, apply the inverse of $r$ to the above array to obtain the desired boat sequence.

More generally, in order to get the corresponding boat sequence for a RTCN on $\ell$ leaves, we first have to reduce it to the RTCN on $2$ leaves by applying Proposition~\ref{rec-rtcn} $\ell-2$ times and retain the operations from the first paragraph for each step. Then, we can start from the (trivial) boat sequence and use the above procedure $\ell-2$ times to construct the corresponding boat sequence; see Figure~\ref{RTCN-to-boat} for an example.

Also, from the above, it is clear that branching events correspond to steps where the person with the maximal rank was picked for the return trip (or, using the language introduced just before Theorem~\ref{only-thm}, the most skilled person on the opposite shore). This proves Theorem~\ref{only-thm} from the introduction.

\section{Branching events, permutations and ranked binary trees}\label{RTCN-fixed-b}
\begin{figure}[t]
\centering
\includegraphics{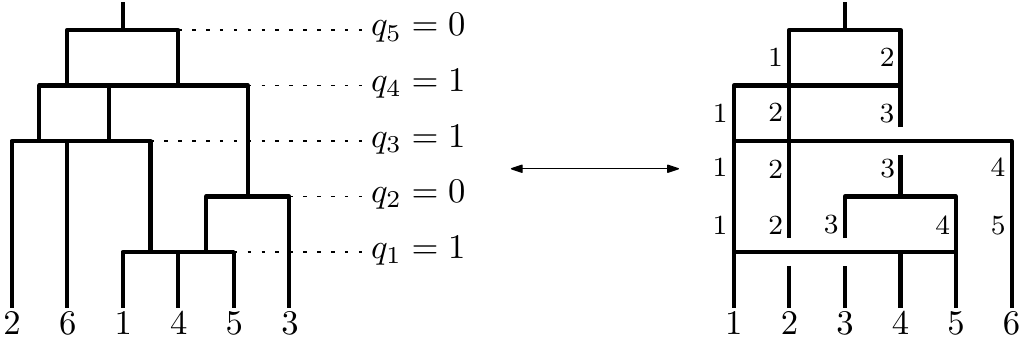}
\caption{\label{fig:RTCN with numbered events} Left: a RTCN with $\ell=6$ leaves and $k=3$ reticulation events; the profile equals $(q_i)_{i=1}^{5}=(1,0,1,1,0)$ and the sequence of positions of reticulation events equals $(a_i)_{i=1}^{3}=(1,3,4)$. Right: the same RTCN redrawn as described in Section~\ref{RTCN-fixed-b}.}
\end{figure}

In this section, we describe an explicit bijection between the set of RTCNs with $\ell$ leaves and $k$ reticulation events (equivalently, with $\ell-1-k$ branching events) and the set of pairs $(T, \sigma)$ where $T$ is a ranked binary tree with $\ell$ leaves and $\sigma$ is a permutation of $\{1,\ldots,\ell-1\}$ consisting of $\ell-1-k$ disjoint cycles.

\paragraph{From a RTCN to a ranked binary tree and permutation.}

Given a RTCN with $\ell$ leaves labeled $1$ to $\ell$, define its profile as the sequence $(q_i)_{i=1}^{\ell-1}\in\{0,1\}^{\ell-1}$ where $q_i=0$ if and only if the $i$-th event from the bottom is a branching event (see left of Figure~\ref{fig:RTCN with numbered events}). Note that we necessarily have $q_{\ell-1}=0$, as only two lineages (i.e., vertical edges) are available to merge for the last event. Suppose $\sum_{i=1}^{\ell-1}q_i=k$, that is, that the RTCN under consideration has $k$ reticulation events, and let $a_1<a_2<\ldots<a_k$ be the indices $i\in \{1,\ldots,\ell-2\}$ such that $q_i=1$, corresponding to the reticulation events in question (note that the reticulation event $a_i$ happens at time $\ell-a_i$, with the notation from Figure~\ref{events}-(b)).

Next, note that we can consider lineages above (and below) each event to be naturally ordered as follows. Assuming that lineages just below event $i$ are labeled $1$ to $\ell-i+1$ and that the event is a branching event, label the $\ell-i$ lineages just above the event with the integers $1$ to $\ell-i$ so that their order reflects the order after the branching event, with the parent lineage identified with its child whose label is smaller. Similarly, assuming $i$ is a reticulation event, order the lineages present just above the event $i$ consistently with their order after the event (ignoring the hybrid lineage resulting from the reticulation event). This is by no means the only possible convention (other conventions would yield equivalent bijections), but we may use it to determine the left-to-right order of lineages when the RTCN is drawn in the plane (so that we can say that the $\ell-i$ lineages above event $i$ are labeled ``left-to-right'' with the numbers $1$ to $\ell-i$). The planar embedding induced by this convention is depicted on the right of Figure~\ref{fig:RTCN with numbered events}: the RTCN on the left is redrawn on the right with lineages ordered left to right according to their current labels, which are explicitly indicated above and below each event; in the representation on the right, parent lineages of reticulation events are drawn directly above the respective non-hybrid child lineages.

In order to obtain a ranked binary tree, we will replace each of the $k$ reticulation events of the RTCN with a branching event; at the same time, we shall keep track of additional information allowing the recovery of the original reticulation event.

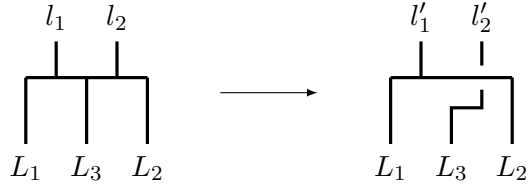
\begin{figure}\centering
\begin{tikzpicture}[xscale=.8, yscale=.8]
\node at (0,0) {$L_1$};
\node at (1,0) {$L_3$};
\node at (2,0) {$L_2$};
\node at (0.5,2.5) {$l_1$};
\node at (1.5,2.5) {$l_2$};
\draw[very thick] (0,.4)--(0,1.5) (1,.4)--(1,1.5) (2,.4)--(2,1.5) (0,1.5)--(2,1.5) (0.5,1.5)--(0.5,2.1) (1.5,1.5)--(1.5,2.1);

\draw[->,>=latex] (3.2,1.25)--(4.8,1.25);

\node at (6,0) {$L_1$};
\node at (7,0) {$L_3$};
\node at (8,0) {$L_2$};
\node at (6.5,2.5) {$l_1'$};
\node at (7.5,2.5) {$l_2'$};
\draw[very thick] (6,.4)--(6,1.5) (7,.4)--(7,1)--(7.5,1)--(7.5,1.3) (8,.4)--(8,1.5) (6,1.5)--(8,1.5) (6.5,1.5)--(6.5,2.1) (7.5,1.7)--(7.5,2.1);
\end{tikzpicture}
\caption{\label{fig:replace reticulation with branching}A reticulation event is replaced by a branching event: the hybrid lineage $L_3$ becomes the future of the former rightmost parent lineage, which is no longer involved in the event. Note that the number $L_3$ need not be between $L_1$ and $L_2$: for examples where it is not, see the second and third transformations performed in Figure~\ref{fig:example of 3-step bijection}.}
\end{figure}

\vspace*{0.1cm}
\begin{figure}[t]\centering
\includegraphics[scale=1]{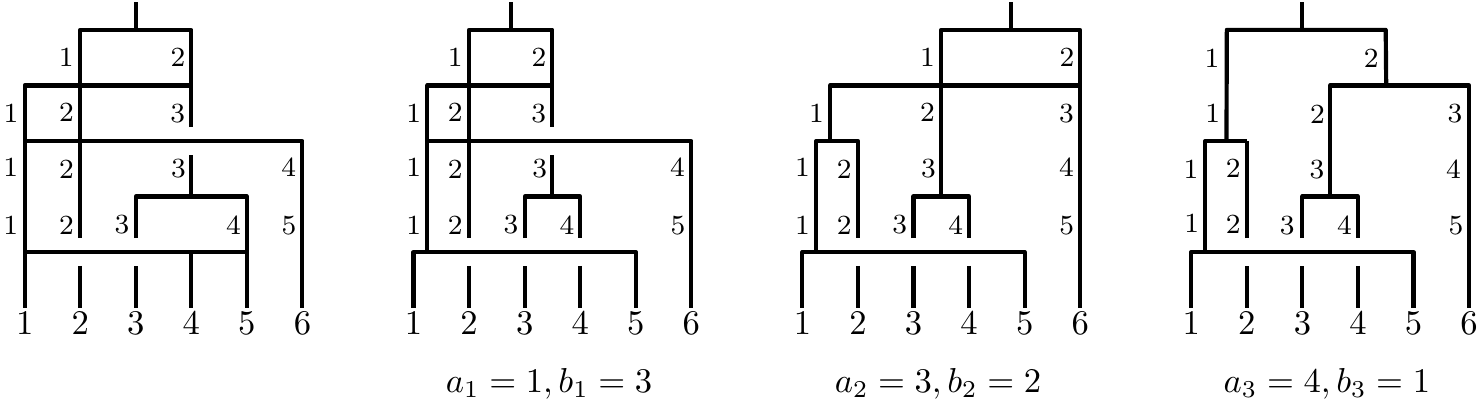}
\caption{\label{fig:example of 3-step bijection}From the RTCN displayed to the far left to its corresponding ranked binary tree obtained on the right, with its accompanying permutation $\sigma=(a_1,a_1+b_1)(a_2,a_2+b_2)(a_3,a_3+b_3)=(1,4)(3,5)(4,5)=(1,5,3,4)(2)$, which is the product of $\ell-1-k=6-1-3=2$ disjoint cycles. The three subsequent steps turn reticulation events 1, 3, 4 into branching events. The values of $(L_1,L_2,L_3,l_1,l_2,l_1',l_2')$ are $(1,5,4,1,4,1,4)$ for the first step, $(1,2,4,1,2,1,3)$ for the second, and $(2,3,1,1,2,2,1)$ for the third.}
\end{figure}

Build the corresponding ranked binary tree to our RTCN by considering in turn the events $a_1,\ldots, a_k$. The result of the reticulation event $a_1$ is given by two lineages $L_1<L_2$, labeled $l_1<l_2$ \emph{above} the reticulation event, and a third hybrid lineage $L_3$; replace this by a branching event that yields $L_1,L_2$ from a single parent lineage $P$ and leaves $L_3$ intact. Some convention is needed in order to match lineages $P, L_3$ to lineages $l_1, l_2$; we identify $P$ with $l_1$ and assign the evolution of the hybrid lineage $L_3$ to the lineage $l_2$ (Figure~\ref{fig:replace reticulation with branching}). Labels assigned to lineages below event $a_1$ are unchanged by this procedure, but labels above the event are recomputed according to the rules we explained previously. In particular, labels immediately above the event $a_1$ range between $1$ and $\ell-a_1$; the lineage that used to be labeled $l_1$ is now labeled $l_1'$ and the lineage that used to be labeled $l_2$ is now labeled $l_2'$ where the values of $l_1', l_2'$ depend on $L_1, L_2, L_3$: if $L_3<L_1$, then $l_1'=L_1=l_1+1$ and $l_2'=L_3$; if $L_2<L_3$, then $l_1'=L_1=l_1$ and $l_2'=L_3-1$; if $L_1<L_3<L_2$, then $l_1'=L_1=l_1$ and $l_2'=L_3$.

The result of this operation is a RTCN $T_1$ with $k-1$ reticulation events. In order to be able to reinstate the original reticulation, it is enough to keep track of $a_1$ (identifying the branching event to be turned back into a reticulation event) and a number $b_1$ that identifies the lineage to turn back into the original hybrid. Note that the label $l_2'$, for example, would allow the recovery of the original RTCN; $l_2'$ is an integer between $1$ and $\ell-a_1$, but cannot take all values in the range independently of $T_1$, as it cannot be equal to $l_1'$, that is, to the label of the parent lineage of the branching event $a_1$. We therefore define the number $b_1$ as $l_2'$ if $l_2'<l_1'$ and $l_2'-1$ otherwise. This way, we have $1\leq b_1\leq \ell-1-a_1$; since $l_1'$ can be determined from $T_1$, the pair $(a_1,b_1)$ is enough to recover $l_2'$ and thus the original RTCN (as we will further discuss later).

We repeat this operation for each reticulation to obtain a final ranked binary tree $T_k$ and a sequence $(a_1,b_1),\ldots,(a_k,b_k)$ where $1\leq b_i\leq \ell-1-a_i$: see Figure~\ref{fig:example of 3-step bijection} for an example.

This sequence of $k$ pairs can be naturally encoded by the permutation $\sigma_k=(a_1,b_1+a_1)\cdot (a_2, b_2+a_2)\cdot \ldots \cdot (a_k,a_k+b_k)$ in the symmetric group $\S_{\ell-1}$ (where we adopt the convention that composition of permutations is obtained via multiplication from the right); this is the product of $k$ transpositions such that $a_1<\ldots<a_k$ and $b_i+a_i>a_i$. The latter condition is actually equivalent to the permutation being a product of $\ell-1-k$ disjoint cycles, as clarified by the following lemma:
\begin{lmm}\label{lemma 1}
For all integers $n>k>0$, a permutation $\sigma\in \S_n$ is the product of $n-k$ disjoint cycles if and only if it can be expressed as a product of transpositions of the form $(x_1,y_1)\cdot\ldots\cdot(x_k,y_k)$ where $0<x_1<\ldots <x_k<n$ and $x_i<y_i\leq n$ for $1\leq i \leq k$; moreover, if such an expression for $\sigma$ exists, then it is unique.
\end{lmm}

\begin{proof}
First of all, given a positive integer $n$, we show the fact that any product of $k<n$ transpositions with the properties described in the statement is the product of $n-k$ disjoint cycles by induction on $k$. Given a positive integer $i<n$, consider a permutation $\tau\in\S_n$ of the form $\tau=(\alpha_1,\beta_1)\cdot\ldots\cdot(\alpha_{i-1},\beta_{i-1})$ where $\alpha_1<\alpha_2<\ldots<\alpha_{i-1}$ and $\alpha_j<\beta_j$ for $1\leq j\leq i-1$; assume that $\tau$ has $n-i+1$ cycles and that each $j$ in $\{1,\ldots,n\}\setminus\{\alpha_1,\ldots,\alpha_{i-1}\}$ is the maximum element of the cycle it belongs to in the factorisation of $\tau$ as a product of disjoint cycles. Consider any permutation of the form $\tau\cdot (\alpha_i,\beta_i)$ where $n\geq \beta_i>\alpha_i>\alpha_{i-1}$. The elements $\alpha_i,\beta_i$ must belong to separate cycles in $\tau$ and each must be the maximum of its own cycle, as neither can be in $\{\alpha_1,\ldots,\alpha_{i-1}\}$. Denote by $c_\alpha$ the cycle of $\tau$ containing $\alpha_i$ and by $c_\beta$ the one containing $\beta_i$; the disjoint cycles of the permutation $\tau\cdot(\alpha_i,\beta_i)$ are those of $\tau$, with the cycles $c_\alpha$ and $c_\beta$ merged into one.

Note that the property that elements not in $\{\alpha_1,\ldots,\alpha_i\}$ are maximal in their cycles holds for $\tau\cdot(\alpha_i,\beta_i)$. Indeed, if some element is not in $\{\alpha_1,\ldots,\alpha_{i-1}\}$ and does not belong to $c_\alpha$ or $c_\beta$, then it is maximal in its cycle in $\tau$, and thus it is maximal in the same cycle of $\tau\cdot(\alpha_i,\beta_i)$. The elements $\alpha_i$ and $\beta_i$, which do not appear in the list $\{\alpha_1,\ldots,\alpha_{i-1}\}$, are the maximums of the cycles $c_\alpha$ and $c_\beta$, respectively. It follows that the element $\beta_i$ is maximal in its cycle of $\tau\cdot(\alpha_i,\beta_i)$, which is the union of $c_\alpha$ and $c_\beta$. Finally, suppose by contradiction that an element of the new cycle other than $\beta_i$ belongs to $\{1,\ldots,n\}\setminus\{\alpha_1,\ldots,\alpha_{i}\}$; this would imply that it was maximal in its cycle in $\tau$, and therefore in $c_\alpha$ or $c_\beta$; it would thus belong to $\{\beta_i,\alpha_i\}$, which we have excluded.

It follows from the fact that the identity of $\S_n$ (the product of zero transpositions) consists of $n$ disjoint cycles of length one (and is thus such that each element in $\{1,\ldots,n\}$ is maximal in its cycle) that any permutation of the form described in the statement is the product of $n-k$ disjoint cycles.

Conversely, consider a permutation $\sigma\in\S_n$ made up of $n-k<n$ disjoint cycles; let us determine positive integers $x_1,\ldots,x_k,y_1,\ldots,y_k$ between $1$ and $n$ such that $x_1<\ldots<x_k$, $x_i<y_i$ for $1\leq i\leq k$, and $\sigma=(x_1,y_1)\cdot\ldots\cdot(x_k,y_k)$. The element $x_1$ must necessarily be the smallest one that is not fixed by $\sigma$. The element $y_1$ must be such that $\sigma(y_1)=x_1$, as $x_1$ is not moved by any of the other transpositions. Consequently, we have $y_1=\sigma^{-1}(x_1)>x_1$. We can repeat this procedure recursively on the permutation $(x_1,\sigma^{-1}(x_1))\cdot \sigma$ (which fixes all elements up to and including $x_1$) to obtain the subsequent transpositions $(x_2,y_2),\ldots,(x_k,y_k)$.
\end{proof}

As a result, from our initial RTCN we have constructed a pair given by a ranked binary tree and a permutation in $\S_{\ell-1}$ which, by applying Lemma~\ref{lemma 1} with $n=\ell-1$, we find to be a product of as many disjoint cycles as the number of branching events in the original RTCN (that is, $\ell-1-k$). We have already sketched how we can revert each change made to recover the original RTCN, but we give a more complete description of the inverse construction below.

\paragraph{From a ranked binary tree and a permutation to a RTCN.} We are given a ranked binary tree with $\ell$ labeled leaves and a permutation $\sigma$ in $\S_{\ell-1}$ with $\ell-1-k$ cycles.

First of all, by Lemma~\ref{lemma 1} we can express $\sigma$ in a unique way as a product of transpositions of the form previously described, as $(x_1,y_1)\cdot\ldots\cdot (x_k,y_k)$ with $0<x_1<\ldots<x_k$ and $x_i<y_i\leq n$ for $1\leq i\leq k$. We then obtain our pairs $(a_i,b_i)_{i=1}^k$ as $a_i=x_i$, $b_i=y_i-x_i$.

All that is left to do is replace the branching events $a_k,\ldots, a_1$ in our ranked binary tree (numbered from the bottom) by reticulation events. We can still number our lineages from $1$ to $\ell-i$ above branching event $i$ (again, think of the order induced by assigning to the parent the label of the smaller child).

For $i=k,\ldots,1$ in turn, consider lineages $1$ to $\ell-a_i$ right above branching event $a_i$ and let $l$ be the label assigned to the parent of the branching event $a_i$; pick the past of lineage $b_i+s$ where $s=1$ if $b_i\geq l$ and $s=0$ otherwise, to become the past of the larger child issued from the branching event (while the past of the branching event becomes the past of the smaller child), and make the future of lineage $b_i+s$ into a hybridised version of the two lineages involved in the branching event.

This operation yields a RTCN and is precisely the inverse of the one previously described (see again Figure~\ref{fig:example of 3-step bijection} where one can follow the steps backwards to recover the RTCN from the final ranked binary tree).

\section{Ranked tree containment and phylogenetic trees}\label{tree-con}

In this final section, we give a bijection between the set of RTCNs which contain a fixed ranked tree and the set of phylogenetic trees; see the last two paragraphs of Section~\ref{intro}.

We first recall how all RTCNs which contain a fixed ranked tree $T$ are constructed from $T$; see \cite{BiLaSt} where this was used to prove (\ref{num-tree-contain}). Starting from the leaves and moving back in time, for every branching event of $T$ a decision is taken whether or not the branching event is turned into a reticulation event and if yes, which of the (future) lineages not in the branching event is connected with which of the (future) lineages from the branching event; see Figure~\ref{ranked-tree}-(b) for an example. We will call the resulting RTCN $N$ and will directly translate each step in the above process of building $N$ from $T$ into a corresponding step for building a phylogenetic tree (time will, however, be reversed).

\vspace*{0.2cm}
\begin{figure}[h!]
\centering
\includegraphics{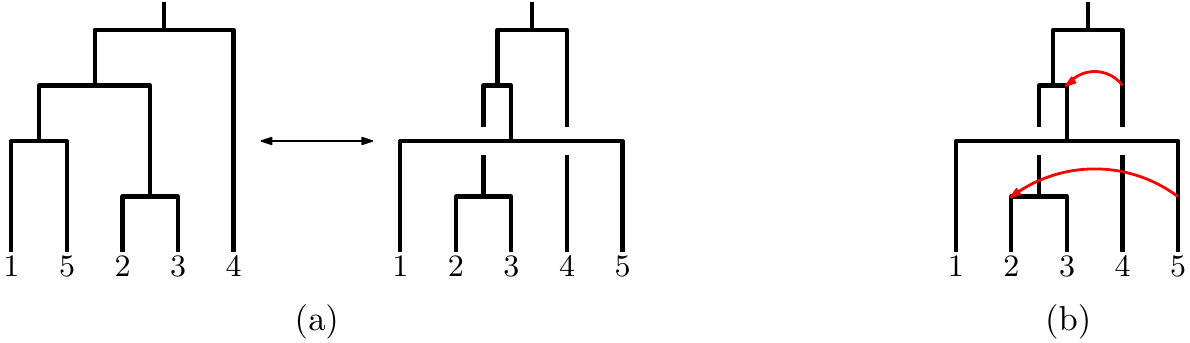}
\caption{(a) A ranked tree $T$ on $6$ leaves which is redrawn such that the leaf labels are in increasing order;  (b) A RTCN $N$ containing the ranked tree.}\label{ranked-tree}
\end{figure}
\vspace*{0.2cm}

In order to describe this, we first draw the ranked tree $T$ so that all the leaves are in increasing order when read from left to right; see Figure~\ref{ranked-tree}-(a). Then, we consider the resulting $N$ (with the same order of the leaves). Now, starting from a tree consisting of a root with one child labeled by $1$, we build the (rooted) tree $\tau$ with leaf labeled by $1,\ldots,\ell$ and internal nodes labeled by $\overline{1},\ldots,\overline{\ell-1}$ where $\ell$ is the number of labels of $T$ and we assume that $1<\cdots<\ell<\overline{1}<\cdots<\overline{\ell-1}$, by processing the events from $N$ top-down. More precisely, for the $k$-th event in $N$, we do the following:
\begin{description}
\item[(a)] If the $k$-th event is a branching event, we insert a node with label $\overline{k}$ into the root edge of $\tau$ and attach to this node a leaf with label $k+1$;
\item[(b)] If the $k$-th event is a reticulation event that was created from the $k$-th branching event of $T$ by connecting the $i$-th (future) lineage (not counting the lineages of the $k$-th branching event), then, we insert a node with label $\overline{k}$ into one of the edges to the children of the internal node with label $\overline{i}$ of $\tau$ where the child with the smaller (larger) label is chosen depending on whether the left or the right lineage of the $k$-th branching event of $T$ was used to form the reticulation event; moreover, we again attach a leaf with label $k+1$ to the inserted node.
\end{description}
See Figure~\ref{grow-T} for a visualization of the the above construction for the RTCN from Figure~\ref{ranked-tree}-(b).

\begin{figure}[t!]
\centering
\includegraphics{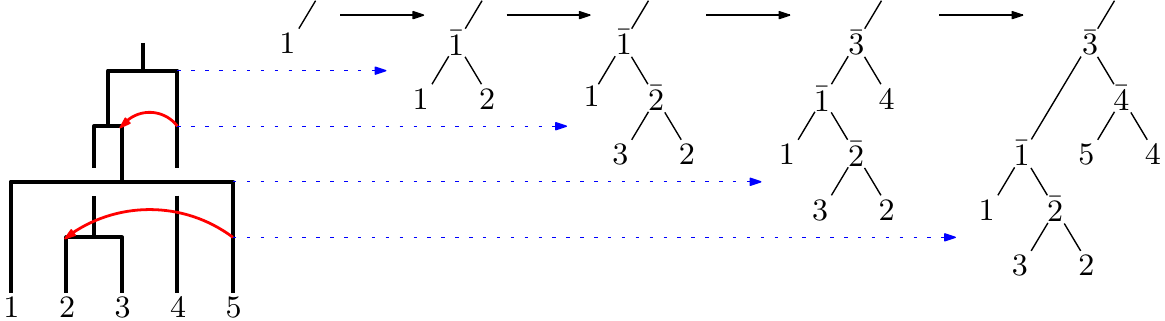}
\caption{The rooted labeled tree constructed from the RTCN from Figure~\ref{ranked-tree}-(b) which contains the ranked tree from Figure~\ref{ranked-tree}-(a).}\label{grow-T}
\end{figure}

Next, we remove the labels of all international nodes and the root edge of $\tau$. The resulting tree is a phylogenetic tree; see the tree on the right of Figure~\ref{bijection}. This tree is the image of $N$.

\vspace*{0.2cm}
\begin{figure}[h!]
\centering
\includegraphics[scale=0.9]{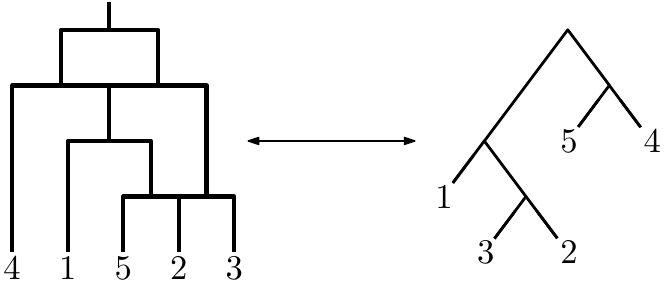}
\caption{The image of the RTCN from Figure~\ref{ranked-tree}-(b) under the bijection.}\label{bijection}
\end{figure}
\vspace*{0.2cm}

Finally, it is not hard to see that the above construction is reversible: e.g., if the phylogenetic tree on the right of Figure~\ref{bijection} is given, find the largest label ($5$); its parent will have label $\overline{4}$. Then, remove the largest label, the parent label and the edge which joins them. Continue until all internal nodes have received a label (where the number of the label is decreased by $1$ in every step). This gives the right-most tree in Figure~\ref{grow-T} which unambiguously encodes the whole tree construction process from Figure~\ref{grow-T}. Clearly, this process can be used to re-construct $N$.

\section*{Acknowledgments} We thank Fran\c{c}ois Bienvenu for putting us into contact. Moreover, we thank the two anonymous referees for a careful reading. The second author acknowledges financial support by the Ministry of Science and Technology, Taiwan under the research grant MOST-109-2115-M004-003-MY2; the third author was partially supported by the same funding agency under the grant MOST-110-2115-M-017-003-MY3.


\begin{thebibliography}{99}
\bibitem{BiLaSt} F. Bienvenu, A. Lambert, M. Steel. Combinatorial and stochastic properties of ranked tree-child networks, {\it Random Struct. Algor.}, in press. https://doi.org/10.1002/rsa.21048
\bibitem{BoGaMa} M. Bouvel, P. Gambette, M. Mansouri (2020). Counting phylogenetic networks of level 1 and 2, {\it J. Math. Biol.}, {\bf 81:6-7}, 1357--1395.
\bibitem{DiSeWe} C. McDiarmid, C. Semple, D. Welsh (2015). Counting phylogenetic networks, {\it Ann. Comb.}, {\bf 19:1}, 205--224.
\bibitem{FuGiMa1} M. Fuchs, B. Gittenberger, M. Mansouri (2019). Counting phylogenetic networks with few reticulation vertices: tree-child and normal networks, {\it Australas. J. Combin.}, {\bf 73}, 385--423.
\bibitem{FuGiMa2} M. Fuchs, B. Gittenberger, M. Mansouri (2021). Counting phylogenetic networks with few reticulation vertices: exact enumeration and corrections, {\it Australas. J. Combin.}, {\bf 82:2}, 257--282.
\bibitem{FuHuYu} M. Fuchs, E.-Y. Huang, G.-R. Yu. Counting phylogenetic networks with few reticulation vertices: a second approach, arXiv:2104.07842.
\bibitem{FuYuZh1} M. Fuchs, G.-R. Yu, L. Zhang (2021). On the asymptotic growth of the number of tree-child networks, {\it European J. Combin.}, {\bf 93}, 103278, 20pp.
\bibitem{FuYuZh2} M. Fuchs, G.-R. Yu, L. Zhang. Asymptotic enumeration and distributional properties of galled networks, arXiv:2010.13324.
\bibitem{HuRuSc} D. H. Huson, R. Rupp, C. Scornavacca. {\it Phylogenetic Networks: Concepts, Algorithms and Applications}, Cambridge University Press, 1st edition, 2010.
\bibitem{SeSt} C. Semple and M. Steel. {\it Phylogenetics}, Oxford University Press, Oxford, 2003.
\bibitem{St} M. Steel. {\it Phylogeny---Discrete and Random Processes in Evolution}, CBMS-NSF Regional Conference Series in Applied Mathematics, 89, Society for Industrial and Applied Mathematics (SIAM), Philadelphia, PA, 2016.
\end{thebibliography}
\end{document}